\documentclass[12pt]{amsart}

\usepackage{amsmath,amssymb,amsthm,enumerate}

\newtheorem{theorem}{Theorem}[section]

\newtheorem{lemma}[theorem]{Lemma}
\newtheorem{proposition}[theorem]{Proposition}

\theoremstyle{definition}

\newtheorem{example}[theorem]{Example}

\theoremstyle{remark}
\newtheorem{remark}[theorem]{Remark}
\newtheorem{question}[theorem]{Question}

\newcommand{\bP}{\mathbb P}
\newcommand{\bC}{\mathbb C}
\newcommand{\bR}{\mathbb R}

\newcommand{\bZ}{\mathbb Z}

\newcommand{\aut}{\mathrm{Aut}}

\DeclareMathOperator{\im}{Im}
\DeclareMathOperator{\re}{Re}

\DeclareMathOperator{\tr}{tr}
\DeclareMathOperator{\ad}{ad}

\begin{document}
\title[]{Umbilical points on three dimensional strictly pseudoconvex CR manifolds. I. Manifolds with $U(1)$-action.}
\author{Peter Ebenfelt}
\address{Department of Mathematics, University of California at San Diego, La Jolla, CA 92093-0112}
\email{pebenfelt@ucsd.edu}
\author{Duong Ngoc Son}
\address{Department of Mathematics, University of California at Irvine, CA 92697-3875}
\email{snduong@math.uci.edu}
\date{\today}
\thanks{The first author was supported in part by the NSF grant DMS-1301282.}
\begin{abstract} The question of existence of umbilical points, in the CR sense, on compact, three dimensional, strictly pseudoconvex CR manifolds was raised in the seminal paper by S.-S. Chern and J. K. Moser in 1974. In the present paper, we consider compact, three dimensional, strictly pseudoconvex CR manifolds that possess a free, transverse action by the circle group $U(1)$. We show that every such CR manifold $M$ has at least one orbit of umbilical points, {\it provided} that the Riemann surface $X:=M/U(1)$ is not a torus. In particular, every compact, circular and strictly pseudoconvex hypersurface in $\bC^2$ has at least one circle of umbilical points. The existence of umbilical points in the case where $X$ is a torus is left open in general, but it is shown that if such an $M$ has additional symmetries, in a certain sense, then it must possess umbilical points as well.
\end{abstract}


\maketitle

\section{Introduction}\label{intro}

Let $M=M^{2n+1}$ be a strictly pseudoconvex CR manifold of dimension $2n+1$, where $n\geq 1$ denotes the CR dimension of $M$. The standard model for such a CR manifold is a sphere $S^{2n+1}$ in $\bC^{n+1}$. The CR structure of $M$, near any point $p\in M$, can be locally approximated by that of a sphere. (The reader is referred to the paper by Chern and Moser \cite{CM74} for precise statements and details.) If $n\geq 2$, then the generic order of this spherical approximation is 3 and the obstruction to higher order approximation is the nonvanishing of the Cartan-Chern-Moser CR curvature tensor $S=S_{\alpha}{}^\beta{}_{\mu\bar\nu}$. If $n=1$, then the generic order of the spherical approximation is 5 and the obstruction to higher order approximation is the nonvanishing of Cartan's 6th order invariant $\bar Q=Q^1{}_{\bar 1}$. Points on $M$ where these obstructions vanish are called (CR) {\it umbilical points}. It is well known \cite{CM74} that $M$ is locally spherical (i.e., locally CR equivalent to a sphere) near a point $p\in M$ if and only if $M$ is umbilical in an open neighborhood of $p$.

The equations for umbilical points for $n\geq 2$ are highly ``overdetermined'', as they are defined by the vanishing of a fourth order curvature tensor on the $n$-dimensional CR tangent space as the points range over a manifold of dimension $2n+1$; whereas in the case $n=1$, the equations are ``under-determined'', being defined by the vanishing of a single complex-valued ``function'' on a three dimensional CR manifold. The question of existence of umbilical points on compact, three dimensional CR manifolds (always assumed to be strictly pseudoconvex in this paper) was explicitly raised in \cite{CM74}. Examples of compact CR manifolds without umbilical points were found by E. Cartan in his classification \cite{Cartan33} of homogeneous, three dimensional CR manifolds. By Cartan's work, the only homogeneous, compact examples are given by
\begin{equation}\label{mualph}
\mu_\alpha:=\{[z_0:z_1:z_2]\in \bC\bP^2\colon |z_0|^2+|z_1|^2+|z_2|^2=\alpha
|z_0^2+z_1^2+z_2^2|\},\quad \alpha >1,
\end{equation}
and their covers, which were subsequently classified in \cite{Isaev06} as a $4\!:\! 1$ cover $\mu_\alpha^{(4)}$ (diffeomorphic to a sphere; see also below) that factors through a $2\!:\! 1$ cover $\mu_\alpha^{(2)}$, consisting of the intersection of sphere and a holomorphic quadric in $\bC^3$. It is an open question whether there exists a compact, three dimensional CR manifold without umbilical points, which embeds in $\bC^2$. (It is known that $\mu^{(2)}_\alpha\subset \bC^3$ does not embed in $\bC^2$, see e.g.\ \cite{Isaev14}, and $\mu_\alpha^{(4)}$ does not embed in any $\bC^{n+1}$ for any $n$, see \cite{Rossi65}.)

Before describing the main results in this paper, we shall mention briefly some known results regarding the existence of (CR) umbilical points in CR geometry, as well as some related results concerning umbilical points (of the second fundamental form) in classical geometry. In the case $n\geq 2$, Webster \cite{Webster00} proved that a generic real ellipsoid (i.e., without spherical sections) in $\bC^{n+1}$ has no umbilical points, which is perhaps not surprising in view of the overdetermined nature of the system of equations defining such points in the higher dimensional case. On the other hand, in the case $n=1$, Huang and Ji \cite{HuangJi07} showed that every real ellipsoid in $\bC^2$ has at least 4 umbilical points. One can also show that on a compact {\it real hypersurface of revolution} (as in \cite{Webster02}) in $\bC^2$, the real curve of points where the $U(1)$-action (induced by revolving) degenerates to fixed points must be umbilical; this is a direct consequence of the fact that there are at most two local CR automorphisms on a three dimensional CR manifold fixing a non-umbilical point \cite{CM74}; note that a generic ellipsoid in $\bC^2$ is not a a surface of revolution. Beyond this, not much is known (at least to the best of the authors' knowledge) regarding existence of umbilical points in CR geometry.

In the theory of surfaces in $\bR^3$, it is a classical conjecture of Carath\'{e}odory that every compact surface bounding a convex region in $\bR^3$ has at least 2 umbilical points, where a point on the surface is umbilical if the two eigenvalues of its second fundamental form are equal (or, equivalently, if the embedded surface can be osculated by a sphere to a higher-than-expected order). There is an extensive literature on this conjecture, with a proof in the real-analytic case by Hamburger \cite{Hamburger40}. While the idea in the proof of Carath\'{e}odory's Conjecture is simple, the details are very complicated. The main general tool is the Poincar\'e-Hopf Index Theorem, which asserts that the sum of the indices at singular points of the foliation by principal line fields (i.e., the umbilical points) equals the Euler characteristic of the surface, which implies that every surface with genus different from one must have umbilical points. (Standard embeddings of the torus have no umbilical points). To establish the Carath\'{e}odory Conjecture, it remains to show the (difficult and still open in general) local result, also known in a slightly different form as Loewner's Conjecture, that the index at an umbilical point does not exceed one. The theory of CR umbilical points on unit circle bundles over Riemann surfaces in the context of classical geometry of surfaces is discussed in Section \ref{CR-classical} below.

The main purpose of this paper is to study (CR) umbilical points on strictly pseudoconvex, three dimensional CR manifolds. We begin by formulating a special case of our main result as an existence theorem for umbilical points on compact, circular real hypersurfaces in $\bC^2$. Recall that a domain $\Omega\subset \bC^2$ is called {\it complete circular} if $z\in \Omega$ implies that the disk
$$D(z):=\{re^{it}z\in \bC^2\colon 0\leq r\leq 1,\ t\in \bR\}$$
is contained in $\Omega$. If $M$ is the boundary of a complete circular domain, then it follows that $U(1)$ acts freely on $M$ by rotations about the origin, $z\mapsto e^{it}z$ for $e^{it}\in U(1)$. In particular, potential umbilical points necessarily appear as circles, orbits of $U(1)$, in $M$.

\begin{theorem}\label{MainCirc} Let $M$ be a smooth, compact hypersurface in $\bC^2$ that bounds a complete circular domain. Then the set of umbilical points on $M$ contains at least one circle.
\end{theorem}

\begin{remark} {\rm We note that a real ellipsoid in $\bC^2$ does not bound a circular domain, unless the ellipsoid degenerates to a sphere. Nevertheless, complete circular domains occur naturally in many contexts, e.g., as the Kobayashi indicatrix, or circular model, of a convex domain; see e.g. \cite{BlandDuchamp91}, \cite{Lempert81}.
}
\end{remark}

We now turn to the main results in this paper in the general context. These concern the existence of umbilical points on compact CR manifolds with a transverse, free CR $U(1)$-action (examples of which are given by the circular hypersurfaces in Theorem \ref{MainCirc}). The question of embeddability and deformations of such CR manifolds were investigated by, e.g., Epstein \cite{Epstein92}, Bland-Duchamp \cite{BlandDuchamp91}, and Lempert \cite{Lempert92}. Thus, let $M$ be a compact, strictly pseudoconvex, three dimensional CR manifold and assume that there is a free action of $U(1)$ on $M$ via CR automorphisms, such that the action is everywhere transverse to the CR tangent spaces of $M$; we shall refer to such an action as a {\it transverse, free} CR {\it $U(1)$-action}. We shall let $X$ denote the smooth compact surface obtained by $\pi\colon M\to X:=M/U(1)$. This Riemann surface can be given a complex structure by $T^{1,0}X:=\pi_* T^{1,0}M$ and $M$ can in fact be identified with the unit circle bundle in a positive holomorphic line bundle over $X$; see \cite{Epstein92} for details. We note that if $p\in M$ is an umbilical point, then the entire $U(1)$-orbit $\pi^{-1}(\pi(p))$ is umbilical. One of our main results is the following:

\begin{theorem}\label{Main1} Let $M$ be a compact, strictly pseudoconvex, three dimensional CR manifold with a transverse, free {\rm CR} $U(1)$-action. If the compact surface $X:=M/U(1)$ is not a torus, then the set of umbilical points contains at least one $U(1)$-orbit.
\end{theorem}

The reader should recognize that Theorem \ref{MainCirc} is a special case of Theorem \ref{Main1}, in which case the Riemann surface $X$ is $\bC\bP^1$ (obtained by blowing up the origin in $\bC^2$; see also \cite{BlandDuchamp91}). The proof of Theorem \ref{Main1} follows by realizing $M$ as a unit circle bundle in a positive line bundle over $X$, as described in \cite{Epstein92}, and applying our second main result, Theorem \ref{PHthm} below, which is a Poincar\'e-Hopf type index theorem for umbilical points on circle bundles over Riemann surfaces. We describe here the result roughly, and refer to Sections \ref{prelim} and \ref{PHsection} for the precise statement, definitions, and further details. Let $M$ be as in Theorem \ref{Main1}. If the umbilical points on $M$ consist of isolated $U(1)$-orbits, $O_1,\ldots, O_n\subset M$, we can define (roughly) an index of the umbilical orbit $O_k$ as follows: Pick a small piece of a surface in $M$ transverse to the umbilical orbit $O_k$, let $S$ be a small, simple closed curve (``circle'') in this surface, circling $O_k$ once, and let the index $\iota(O_k)$ be $-1/2$ of the topological degree of $Q/|Q|\colon S\to S^1$, where $Q$ denotes (the complex conjugate of) Cartan's 6th order tensor (whose zero locus defines the umbilical set). The result, Theorem \ref{PHthm} below, states that the following formula holds:
\begin{equation}\label{PHformula0}
\sum_{k=1}^n \iota(O_k)=\chi(X),
\end{equation}
where $\chi(X)=2-2g$ denotes the Euler characteristic of $X=M/U(1)$. Theorem \ref{Main1} of course follows from this result. As in the case of Loewner's Conjecture in classical geometry, a local analysis of the indices of the umbilical orbits could potentially provide a more precise estimate for the number of umbilical orbits on $M$ in terms of the genus of $X$. For instance, if the index of an umbilical orbit is always bounded above by one (as predicted in Loewner's Conjecture), then formula \eqref{PHformula0} would imply that a circular hypersurface as in Theorem \ref{MainCirc} has at least two umbilical orbits. This would correspond to the Carath\'{e}odory Conjecture in classical geometry. We should mention that the authors do not know if the foliation associated with CR umbilical points is ``Hessian'' in a sense that would make Loewner's Conjecture itself relevant to the local study. This is discussed and explained in more detail in Section \ref{Loewner} below; see in particular Question \ref{QA} regarding the possible Hessian nature of the CR umbilical foliation. We note, however, that for a ``generic'' umbilical orbit, the index will be $\pm1/2$ (c.f.\ e.g.\ subsection \ref{extrinsic} below) and we may conclude from formula \eqref{PHformula0} that a generic circular hypersurface as in Theorem \ref{MainCirc} should have at least 4 umbilical orbits.

We note that formula \eqref{PHformula0} leaves open the possibility that there are CR manifolds $M$ as in Theorem \ref{Main1} without umbilical points, {\it provided} that the Riemann surface $X$ is a torus. The authors do not know if there are such $M$ without umbilical points. However, we can show that if $M$ also possesses certain additional symmetries, then there must be umbilical points on $M$, even when $X$ is a torus; this is the content of Theorem \ref{Maintorus} below. We shall denote by $\aut(M)$ the group of CR automorphisms of a CR manifold $M$. It is well known that $\aut(M)$, for a compact, strictly pseudoconvex CR manifold $M$ is a Lie group (this follows from the work of Cartan and Chern-Moser in this case; cf. also, e.g., \cite{LMZ08}). Our standing assumption that $M$ has a $U(1)$-action implies in particular that $\dim_\bR\aut (M)\geq 1$. Our final result stated in this introduction is the following:

\begin{theorem}\label{Maintorus} Let $M$ be a compact, strictly pseudoconvex, three dimensional CR manifold with a transverse, free {\rm CR} $U(1)$-action. If $\dim_\bR\aut (M)\geq 2$, then the set of umbilical points contains at least one $U(1)$-orbit.
\end{theorem}

Let us conclude this introduction by discussing the standing assumption in this paper that the compact CR manifold $M$ has a $U(1)$-action that is {\it free} and {\it transverse}. As mentioned above, if the CR manifold $M$ has an effective $U(1)$-action with fixed points (i.e., the action is not free), then these fixed points are necessarily umbilical, since the action yields a one-parameter subgroup of the stability group at such points. Consequently, in terms of proving existence of umbilical points, such existence follows from the existence of fixed points. Of course, it would be of interest to know if there must be other umbilical points as well, in addition to the fixed points.

In contrast to the condition of freeness, the condition that the $U(1)$-action is transverse is more subtle, as is indicated by the following example; the addition of which was inspired by a conversation with H. Jacobowitz \cite{Jacobowitz15}:

\begin{example}\label{su(2)} The family, parametrized by $\alpha>1$, of $4\!:\! 1$ covers $\mu_\alpha^{(4)}$ of Cartan's homogeneous examples \eqref{mualph} of compact, three dimensional CR manifolds without umbilical points can be described as a family of non-standard CR structures on the unit 3-sphere $S^3\subset \bC^2$, defined by their covering maps $\mu_\alpha^{(4)}\to \mu_\alpha$ (see, e.g., \cite{Isaev06}; cf. also \cite{Jacobowitz15}). For every $\alpha>1$, the standard action of $SU(2)$ on $S^3\subset \bC^2$ is a free action on $\mu_\alpha^{(4)}$ by CR automorphisms. If we identify $U(1)$ as a subgroup in $SU(2)$ via
$$
e^{it}\mapsto
\begin{pmatrix} e^{it} & 0 \\ 0 & e^{-it}
\end{pmatrix},
$$
then we obtain a free CR $U(1)$-action on $\mu_\alpha^{(4)}$. However, this action is {\it not transverse} to the CR tangent spaces of $\mu_\alpha^{(4)}$ along a 2-torus in $S^3$. The quotient $M/U(1)$ is still a compact Riemann surface, topologically equivalent to the 2-sphere. Thus, since $\mu_\alpha^{(4)}$ is non-umbilical at every point, we conclude that Theorem \ref{Main1} does not hold without the assumption that the free $U(1)$-action is transverse. In addition, since $\dim_\bR \aut(\mu_\alpha^{(4)})\geq 2$, we also note that the same is true of Theorem \ref{Maintorus}.
\end{example}

This paper is organized as follows. In Section \ref{prelim}, we introduce Cartan's tensor $\bar Q$ in the context of unit circle bundles over Riemann surfaces. In Section \ref{PHsection}, we state and prove the index formula (Theorem \ref{PHthm}) for umbilical circles alluded to above. In Section \ref{CR-classical}, we discuss CR umbilical points on circle bundles in the context of eigenvector field foliations of Hessians on Riemann surfaces, with a separate discussion of the situation when the Riemann surface is a torus in Section \ref{Torus}. The proofs of the results stated in this introduction are then given in Section \ref{Proofs}.

\section{Preliminaries}\label{prelim}

\subsection{Strictly pseudoconvex unit circle bundles} As mentioned in the introduction, any compact strictly pseudoconvex, three dimensional CR manifold $M$ can be realized as the unit circle in a positive holomorphic line bundle over a compact Riemann surface $X$, in view of a construction due to Epstein \cite{Epstein92}. The reader is referred to this paper (especially its appendix A) for that construction. For the remainder of this paper, we shall consider only such unit circle bundles. We note here that for the considerations in this paper (existence of umbilical points) the orientation of the CR manifold $M$ is not important, and we could work equally well in the dual, negative line bundle.

Thus, let $X$ be a Riemann surface (complex manifold of dimension one) and $\pi\colon L\to M$ a positive holomorphic line bundle over $X$. We shall choose a positively curved metric $(\cdot,\cdot)$ in $L$ and let $M$ denote the unit circle bundle in $L$ with respect to this metric, i.e., $M=\{(x,\ell)\in L\colon |\ell|^2_x=1\}$. If $s_0\colon U\subset X\to L$ is a nonvanishing local holomorphic section, then in the induced local trivialization $L|_U\cong U\times\bC$ with coordinates $(z,\tau)\in U\times \bC$, the three dimensional CR manifold $M$ is given by
\begin{equation}\label{Meq}
|\tau|^2h(z,\bar z)=1,
\end{equation}
where $h(z,\bar z)=|s_0|_z^2$. The assumption that the curvature is positive,
\begin{equation}\label{Lcurv}
i \Theta:=-i \partial\bar\partial \log h>0,
\end{equation}
means that $M$ is strictly pseudoconvex. (We mention again that the orientation of $M$ is irrelevant in our context; negative curvature would work equally well.)  If we use polar coordinates $\tau=re^{it}$ in the fibers and $(z,t)\in \bC\times \bR$ as local coordinates on $M$, then
\begin{equation}\label{thetahat}
\hat \theta=-dt+\frac{i}{2}(\partial \log h-\bar\partial \log h)
\end{equation}
is a contact form on $M$ that is compatible with the CR structure, and
\begin{equation}
d\hat \theta=\frac{i}{2}(\bar\partial \partial \log h-\partial\bar \partial \log h)=-i\partial\bar \partial \log h.
\end{equation}
We shall use the notation
\begin{equation}\label{Dnotation}
D:=\frac{\partial}{\partial z},\quad \Delta:=4D\bar D,
\end{equation}
so that
$$
d\hat\theta=-iD\bar D\log h\,  dz\wedge d\bar z=ia^{-1}\,  dz\wedge d \bar z
$$
where $a=a(z,\bar z)$ is the function
\begin{equation}\label{a}
a:=(-D\bar D\log h)^{-1}=\left(-\frac{1}{4}\Delta \log h\right)^{-1}>0.
\end{equation}
Thus, with
\begin{equation}\label{theta0}
\theta_0:=a\hat\theta,
\end{equation}
we have
\begin{equation}\label{dtheta_0}
d\theta_0=i\, dz\wedge d\bar z +\frac{da}{a}\wedge \theta_0.
\end{equation}
The coframe $(\theta_0,dz,d\bar z)$ defines the CR structure on $M$ in the standard way and any other CR coframe $(\theta,\theta^1,\theta^{\bar 1})$ is of the form
\begin{equation}\label{CRcoframes}
\theta=|\lambda|^2\theta_0,\quad \theta^1=\lambda(dz+\mu\theta_0).
\end{equation}
The 1-forms $(\omega,\omega^1,\omega^{\bar 1},\phi)=(\theta_0,dz,d\bar z,da/a)$ yield a section of the Cartan-Chern-Moser (CCM) CR bundle \cite{CM74} that can be used to pullback Cartan's ``6th order'' invariant $\bar Q=Q^1{}_{\bar 1}$. We shall in fact prefer to work with the complex conjugated invariant $Q$. The direct computation in \cite{JacobowitzBook} in a coframe $(\theta,\theta^1,\theta^{\bar 1})$ given by \eqref{CRcoframes} shows that (see pp. 126 and 140 in \cite{JacobowitzBook}; but mind the complex conjugation)
\begin{equation}\label{Q}
Q=-\frac{r}{\lambda^3{\bar\lambda}},
\end{equation}
where $r={r(z,\bar z)}$ is a function explicitly computed from the function
\begin{equation}\label{barb}
\begin{aligned}
q &=-Da/a=-D\log a\\
& = D\log\, (-D\bar D\log h).
\end{aligned}
\end{equation}
In fact, $r$ is obtained by applying a third order differential operator (PDO) to $q$ (see \cite{JacobowitzBook}, eq. (47) on p. 126):
\begin{equation}\label{r}
r=D^2\bar D q-3q D\bar D q +2q^2Dq -Dq\bar Dq.
\end{equation}
Recall that {\it umbilical points} on $M$ are defined to be those where $Q=0$, or equivalently $r=0$. The expression \eqref{r} for identifying the umbilical locus on a rigid hypersurface (i.e., one with an infinitesimal transverse symmetry) in $\bC^2$ was also derived in \cite{Loboda97} by using the extrinsic normal form of Chern--Moser \cite{CM74}; see also below. We observe that the function $r=r(z,\bar z)$ is independent of the fiber variable $t$, and therefore the umbilical locus of $M$ will consist of circles $\pi^{-1}(z_0)$ for $z_0\in X$ at which $r=0$ (where by a slight abuse of notation we also denote by $\pi=\pi|_M$, the restriction of $\pi$ to $M$). We say that $M$ has an {\it isolated umbilical circle} at $z_0\in X$ if $r$ has an isolated zero at $z_0$; we shall refer to $z_0$ as the base of the umbilical circle $\pi^{-1}(z_0)$. Conversely, we shall say that $M$ is {\it totally umbilical} on an open subset $U\subset M$ if $r$ vanishes identically on $\pi(U)$. It is well known \cite{CM74} that $M$ is totally umbilical on $U$ if and only if $M$ is locally spherical on $U$, i.e., $M$ is locally CR diffeomorphic to a standard sphere near every point in $U$.

We shall digress momentarily to discuss umbilical points locally on an embeddable (here, real-analytic for simplicity) three dimensional CR manifold from an extrinsic point of view.

\subsection{Umbilical points and normal forms}\label{extrinsic} A real-analytic strictly pseudoconvex hypersurface $M$ through the origin in $\bC^2$ with coordinates $(z,w)$ can be expressed locally by a defining equation of the form
\begin{equation}\label{NormForm}
\im w=F(z,\bar z,\re w)=|z|^2+\sum_{k,l\geq 2} a_{k l}(\re z)z^k\bar z^l,\quad a_{kl}=\overline{a_{l k}}.
\end{equation}
It is shown in the seminal paper by Chern and Moser \cite{CM74} that such an equation can be brought to a normal form where $a_{22}(u)=a_{32}(u)=a_{33}(u)\equiv 0$. In this case, the coefficient $a_{42}(0)$ represents (the complex conjugate of) Cartan's tensor $Q$ at $0$. If $M$ has an infinitesimal CR automorphism transverse to $T^{1,0}_pM$ near $p$ (such as, e.g., a $U(1)$-action as considered in this paper), then $M$ is called {\it rigid} and there are local coordinates $(z,w)$, vanishing at $p$, such that $M$ is locally defined by \eqref{NormForm} with $F(z,\bar z,u)=F(z,\bar z)$ independent of $u=\re w$. It may still be that in Chern-Moser normal form, the so normalized $F$ depends on $u=\re w$, but as explained in \cite{Loboda97}, one may still achieve a rigid normalization (in which we need not have the Chern-Moser normalization $a_{22}=a_{32}=a_{33}=0$)
\begin{equation}\label{NormForm1}
\im w=F(z,\bar z)=|z|^2+\sum_{k,l\geq 2} a_{k l}z^k\bar z^l,\quad a_{kl}=\overline{a_{l k}}
\end{equation}
such that $a_{42}$ represents $Q$ at $p=0$. Moreover, it is shown that if we set $$
q=q(z,\bar z):=\frac{F_{zz\bar z}}{F_{z\bar z}},
$$
then $Q=Q(z,\bar z)$ is represented by $r=r(z,\bar z)$, given by the differential operator in \eqref{r}, and hence the umbilical points are given by the vanishing of this expression. Since $r$ is independent of $u=\re w$, the umbilical points appear locally as unions of lines (orbits of the infinitesimal CR automorphism) $u\mapsto (z_0,u)$, where $r(z_0,\bar z_0)=0$. If $z_0=0$ is an isolated zero of $r(z,\bar z)$, then the line $\ell_0\subset M$ parametrized by $u\mapsto (0,u)$ is umbilical and one can define the index of this umbilical line as
\begin{equation}\label{index0}
\iota(\ell_0):=-\frac 12\, \deg_0\left(\frac{r}{|r|}\right),
\end{equation}
where $\deg_0$ denotes the topological degree of the mapping $r/|r|$ from a small circle centered at $z=0$ (so small that $z=0$ is the only zero of $r$ inside) to the unit circle. An equivalent definition of the index will be given in the next section in the context of a circle bundle over a Riemann surface, where the reason for the factor 1/2 becomes clear. If $z=0$ is an isolated zero of $r(z,\bar z)$, then it is easy to see that
\begin{equation}\label{indexint}
-\deg_0\left(\frac{r}{|r|}\right)=\frac{i}{2\pi }\int_{\partial D_\epsilon(0)} d\log r,
\end{equation}
where $D_\epsilon(0)$ is the disk $\{|z|<\epsilon\}$; $z=0$ is the only zero of $r$ in $D_\epsilon(0)$; a continuous branch of the logarithm is chosen along the boundary $\partial D_\epsilon(0)$; and the integral is taken with respect to the positive orientation of the boundary. From this we deduce that if $\ell_0$ is an isolated umbilical line with non-zero index, then any sufficiently small perturbation of $r$ will still have zeros in $D_\epsilon(0)$, and therefore $M$ will have umbilical lines close to $\ell_0$. It is also easy to see that if $r(z_0,\bar z_0)=0$, then this is an isolated zero with index $\iota(\ell_0)=-1/2$ if $|r_z|>|r_{\bar z}|$ at $z_0$ and $\iota(\ell_0)=1/2$ if the reverse inequality $|r_z|<|r_{\bar z}|$ holds. We conclude that if $r$ has an isolated zero at $z=0$, with a non-zero index $\iota(\ell_0)$, then a sufficiently small, generic perturbation of $r$ will have isolated zeros of indices $\pm 1/2$ in $D_\epsilon(0)$ adding up to $\iota(\ell_0)$. Further discussion of the local theory can be found, in a different context, in Section \ref{Loewner} below.

We end this digression into the extrinsic point of view to note that there is a $1\!:\!1$ correspondence between CR manifolds that are locally a unit circle bundle over a Riemann surface, given by the equation \eqref{Meq}, and those that can be given by a rigid local equation of the form \eqref{NormForm1}, namely via $\tau=e^{-iw/2}$ and $F=-\log h$ (possibly with some further normalization to get $F(z,\bar z)=|z|^2+O(|z|^4)$).

\section{An index formula for isolated umbilical circles on compact Riemann surfaces}\label{PHsection}

If we examine the transformation rule \eqref{Q} for the Cartan tensor $Q$ as we vary the CR coframe \eqref{CRcoframes}, we notice that there is a choice of a smooth, non-vanishing function $\lambda$ that makes $Q$ identically one near a {\it non-umbilical} point $p\in M$. Such a choice of $\lambda$ yields a contact form $\theta=|\lambda|^2\theta_0$ that satisfies
\begin{equation}\label{dtheta-1}
\begin{aligned}
d\theta &=|\lambda|^2 d\theta_0+\left(\frac{d\lambda}{\lambda}+\frac{d\bar \lambda}{\bar \lambda}\right)\wedge \theta\\
&=
i|\lambda|^2 dz\wedge d\bar z+\left(d\log a+d\log|\lambda|^2\right)\wedge \theta\\
&=i(\lambda dz)\wedge \overline{\lambda dz}+d\log a|\lambda|^2\wedge \theta.
\end{aligned}
\end{equation}
Thus, with
\begin{equation}\label{theta1}
\theta^1=\lambda (dz+i\bar D\log a|\lambda|^2 \theta_0),
\end{equation}
we have
$$
d\theta=i\theta^1\wedge \theta^{\bar 1},
$$
and the computations in \cite{JacobowitzBook} show that the pullback of the Cartan invariant $Q$ using the CCM section of 1-forms $(\theta,\theta^1,\theta^{\bar 1},\phi=0))$ satisfies $Q=1$. We note that $(\theta^1,\theta^{\bar 1})$ is an admissible coframe for the pseudohermitian structure (\cite{Webster78}) defined on $M$ by $\theta$.
We further observe that the function $\lambda=\lambda(z,\bar z)$ satisfies
\begin{equation}\label{lambda}
\bar \lambda\lambda^3=|\lambda|^2\lambda^2=-r(z,\bar z),
\end{equation}
which implies that $\lambda$ is unique up to parity. Since $\lambda=\lambda(z,\bar z)$ is independent of the fiber variable $t$, we may also define a $(1,0)$-form near $z_0=\pi(p)$ on $U\subset X$ by $\xi^1:=\lambda dz$. This form is unique up to parity. By allowing $\lambda$ to vanish when $r$ does (i.e., at the base of umbilical circles), the equation \eqref{lambda} defines $\lambda$ as a possibly non-smooth function with square root branching where $r$ vanishes, and consequently we can extend $\xi^1$ as a double-valued $(1,0)$-form that vanishes when $\lambda$ does. We obtain in this way a well-defined global section $\Xi:=\xi^1\otimes\xi^1$ of the square of the canonical bundle $\kappa_X\otimes \kappa_X$. The section $\Xi$ vanishes precisely at the points $z_0$ where $r(z_0,\bar z_0)=0$, i.e., at the base of the umbilical circles on $M$.

If $M$ has an isolated umbilical circle at $z_0\in X$, we shall define the {\it index of the umbilical circle} $\pi^{-1}(z_0)$ to be
\begin{equation}\label{index}
\iota(z_0):=-\frac{1}{2}\deg_{z_0}\left (\frac{\alpha}{|\alpha|}\right),
\end{equation}
where $\xi^{1}\otimes\xi^{1}=\alpha\, d\zeta\otimes d\zeta$ in some local coordinate $\zeta$ near $z_0$ and $\deg_{z_0}$ refers to the topological degree of $\alpha/|\alpha|$ as a map of a small circle $S^1_\epsilon(z_0):=\{|z-z_0|=\epsilon\}$ to the unit circle $S^1$; the index is easily seen to be independent of the local coordinate used. In the coordinates $z$ used above, we have of course $\alpha=\lambda^2$, and the index counts the number of times $\lambda/|\lambda|$ wraps around $S^1$ (in the negative direction) as $S^1_\epsilon(z_0)$ is traversed (in the positive direction), which could be a half-integer. We have the following ``Poincar\'e-Hopf type" index formula for isolated umbilical circles on a compact circle bundle:

\begin{theorem}\label{PHthm}
Let $X$ be a compact Riemann surface, $\pi\colon  L\to X$ a holomorphic line bundle with a positively curved metric, and $M$ the corresponding, strictly pseudoconvex unit circle bundle in $L$. If $M$ has only isolated circles of umbilical points, based at $z_1,\ldots,z_n\in X$, then
\begin{equation}\label{PHformula}
\sum_{k=1}^n \iota(z_k)=\chi(X),
\end{equation}
where $\chi(X)=2-2g$ denotes the Euler characteristic of $X$ and $\iota(z_k)$ the index of the umbilical circle at $z_k$.
\end{theorem}

\begin{proof}
Let $g$ be a smooth metric on $X$, and let $\Theta$ denote the curvature form of this metric. By the Gauss--Bonnet formula, we have
\begin{equation}\label{GBformula}
\chi(X)=\frac{i}{2\pi}\int_X\Theta.
\end{equation}
The metric $g$ induces a hermitian metric on the square of the canonical bundle $\kappa^2:=\kappa_X\otimes \kappa_X$, whose curvature form is $\Theta_{\kappa^2}=-2\Theta$. Let $D_\epsilon(z_k)$ denote small disks of radius $\epsilon>0$ centered at the bases $z_k$ of the umbilical circles and note, by \eqref{GBformula} and continuity of $\Theta_{\kappa^2}$, that
\begin{equation}\label{GBformula2}
-2\chi(X)=\frac{i}{2\pi}\int_X\Theta_{\kappa^2}=\lim_{\epsilon\to 0} \frac{i}{2\pi}\int_{X_\epsilon}\Theta_{\kappa^2}.
\end{equation}
where $X_\epsilon:=X\setminus \bigcup_{k=1}^n D_{\epsilon}(z_k)$. Let us choose $\epsilon$ so small that the disks $D_{\epsilon}(z_k)$ are disjoint and each closed disk is contained in an open set $U\subset X$ with a local trivialization $L|_U\cong U\times\bC$ as in Section \ref{prelim} above. Recall that $\Xi:=\xi^1\otimes\xi^1$  is a global, non-vanishing section of $\kappa^2$ over $X_\epsilon$, and let $\theta_\Xi$ denote the metric connection form with respect to this basis over $X_\epsilon$. Then, since $\Theta_{\kappa^2}=d\theta_\Xi$, we have
\begin{equation}\label{e1}
\frac{i}{2\pi}\int_{X_\epsilon}\Theta_{\kappa^2}=-\sum_{k=1}^n\frac{i}{2\pi}\int_{\partial D_\epsilon(z_k)}\theta_{\Xi},
\end{equation}
where the path integrals $\int_{\partial D_\epsilon(z_k)}$ are taken in the positive direction.
Pick a $k$ and recall that $\overline{D_\epsilon(z_k)}$ is contained in an open set $U$ on which we have a local coordinate $z$ and a local trivialization $L|_U\cong U\times\bC$. Thus, in $U$ we have $\Xi=\lambda^2 dz\otimes dz$ and if $\theta_0$ denotes the connection form of $\kappa^2$ with respect to $dz\otimes dz$, then by the standard change of basis formula for the connection in a line bundle we have
\begin{equation}
\theta_{\Xi}=\theta_0+d\log \lambda^2.
\end{equation}
We conclude that
\begin{equation}\label{e2}
\begin{aligned}
\frac{i}{2\pi}\int_{\partial D_\epsilon(z_k)}\theta_{\Xi} &=\frac{i}{2\pi}\int_{\partial D_\epsilon(z_k)}\theta_{0}+\frac{i}{2\pi}\int_{\partial D_\epsilon(z_k)}d\log \lambda^2\\
&=\frac{i}{2\pi}\int_{D_\epsilon(z_k)}\Theta_{\kappa^2}+\frac{i}{2\pi}\int_{\partial D_\epsilon(z_k)}d\log \lambda^2,
\end{aligned}
\end{equation}
since $dz\otimes dz$ is a non-vanishing, smooth section of $\kappa^2$ in $U$, and hence, $\Theta_{\kappa^2}=d\theta_0$ in $U$. It is also readily seen that when $z_0$ is the only zero of $\lambda^2$ in $\overline{D_\epsilon(z_k)}$, then
\begin{equation}\label{e3}
\frac{i}{2\pi}\int_{\partial D_\epsilon(z_k)}d\log \lambda^2=-\deg_{z_k}\left(\frac{\lambda^2}{|\lambda^2|}\right)=2\iota(z_k).
\end{equation}
By substituting \eqref{e3} and \eqref{e2} into \eqref{e1} and using that fact that, by continuity,
$$
\lim_{\epsilon \to 0} \int_{D_\epsilon(z_k)}\Theta_{\kappa^2}=0,
$$
we conclude that
\begin{equation}
\frac{i}{2\pi}\int_{X_\epsilon}\Theta_{\kappa^2}=\sum_{k=1}^n -2\iota(z_k),
\end{equation}
which completes the proof of the theorem in view of \eqref{GBformula2}.
\end{proof}

\section{CR umbilics on circle bundles as umbilics of the Hessian of the Gauss curvature}\label{CR-classical}

In this section, we shall relate the (CR) umbilical locus on the unit circle bundle $M$ in the holomorphic line bundle $\pi\colon L\to X$ with positively curved metric $h=(\cdot,\cdot)$ to umbilical points of the Hessian of the Gauss curvature $K$ of the induced metric with K\"ahler form $i\Theta=-i\partial\bar\partial \log h$ on the Riemann surface $X$. Recall that umbilical points on $M$ are defined to be those where Cartan's tensor $Q$ vanishes, or equivalently $r=0$, where $r=r(z,\bar z)$ is given by \eqref{r}. The fact that $r$ is independent of the fiber variable $t$ means, as mentioned above, that the umbilical locus of $M$ will consist of circles $\pi^{-1}(z_0)$ for $z_0\in X$ at which $r=0$. It will actually be convenient to express $r$ as a fourth order partial differential operator $P$ applied to the function $u=-\log a$,
\begin{equation}\label{r2}
r=Pu:=D^3\bar D u-3(Du) D^2\bar D u +2(Du)^2D^2u -(D^2u)D\bar D u,
\end{equation}
where we also recall $a=(-D\bar D\log h)^{-1}$ so that in fact $u=\log (-D\bar D\log h)$. Our first observation is that $r$ is essentially a second covariant derivative of the (Gauss) curvature $K=K^u$ of the metric
\begin{equation}\label{ds}
ds^2:=e^{2\phi}|dz|^2,\quad 2\phi:=u,
\end{equation}
where as usual we have
\begin{equation}\label{K}
K:=-e^{-2\phi}\Delta\phi=-2e^{-u}D\bar Du.
\end{equation}
For a smooth, real-valued function $f$, we shall denote by $f_{;z}$ and $f_{;zz}$ the first and second order covariant derivatives with respect to $z$ (in the $(1,0)$ direction) in the unitary coframe $e^\phi dz$; i.e.,
\begin{equation}\label{Kzz}
f_{;z}=e^{-\phi}Df,\quad f_{;zz}=e^{-2\phi}(D^2f-2(D\phi)Df)=e^{-u}(D^2f-(Du)Df).
\end{equation}
A straightforward computation using \eqref{K} and \eqref{Kzz} proves:

\begin{proposition}\label{KzzProp} If $K=K^u$ denotes the curvature of the metric given by \eqref{ds} and $P$ the fourth order PDO yielding the Cartan invariant $Q$ via \eqref{Q} and \eqref{r2}, then
\begin{equation}\label{Q=Kzz}
Pu=-\frac{e^{2u}}{2}K_{;zz}.
\end{equation}
\end{proposition}
Since $K_{;zz}$ can be written $K_{;zz}=D(e^{-u}DK)$, we also obtain the ``divergence form'' of~$P$:
\begin{equation}\label{divP}
Pu=e^{2u}D(e^{-u}D(e^{-u}D\bar Du)).
\end{equation}

\subsection{Umbilical points of Hessians of functions on a Riemann surface} We shall discuss briefly Hessians of smooth functions on a Riemann surface $X$ to offer a slightly different perspective on CR umbilical points on circle bundles $M$ over $X$. Let $ds^2$ be a smooth metric on $X$, given in a local chart $z=x+iy$ by
\begin{equation}\label{metric}
ds^2=e^{2\phi}|dz|^2=e^{2\phi}(dx^2+dy^2).
\end{equation}
Let $\nabla=\nabla^\phi$ denote the metric (Levi-Civita) connection, $\nabla^{\uparrow}$ the gradient operator on functions $f$,
$$
\nabla^{\uparrow}f=e^{-2\phi}\left(f_x\frac{\partial}{\partial x}+f_y\frac{\partial}{\partial y}\right),
$$
where we have used the notation $f_x$ for the partial derivative $\partial f/\partial x$. The Hessian of $f$, with respect to the metric connection, is now defined by
$$
H^\phi_f:=\nabla\nabla^{\uparrow}f,
$$
which can be viewed as a linear operator on the real tangent space $TX\to TX$, or a $(1,1)$-tensor, which can be represented in the unitary coframe $\{dx\otimes\partial/\partial x, dx\otimes\partial/\partial y, dy\otimes\partial/\partial x, dy\otimes\partial/\partial y\}$ as a symmetric $2\times2$ matrix
\begin{equation}\label{Hf}
H_f:=H^\phi_f=\left(
\begin{matrix} f_{11} & f_{12}\\ f_{12} & f_{22}\end{matrix}
\right).
\end{equation}
The umbilical points of the tensor $H_f$ are those where the two eigenvalues of the matrix \eqref{Hf} are equal (where the foliation by eigenvector fields is singular), which coincides with the points where the expression
$$
f_{11}-f_{22}-2if_{12}
$$
vanishes. A straightforward calculation shows that this expression essentially coincides with the second covariant derivative of $f$ in the $z$-direction,
\begin{equation}\label{q}
f_{;zz}=\frac{1}{4}(f_{11}-f_{22}-2if_{12}),
\end{equation}
and hence the set of umbilical points of $H_f$ consists of the locus where $f_{;zz}=0$. We note that
\begin{equation}\label{Xi}
\Xi_f:=f_{;zz} (e^{\phi}dz\otimes e^{\phi}dz)=e^{2\phi}f_{;zz}\, dz\otimes dz,
\end{equation}
defines a global section of the square of the canonical bundle $\kappa_X\otimes \kappa_X$ on $X$. In complete analogy with Section \ref{PHsection}, we may now define the index of an umbilical point $z_0\in X$ of the Hessian $H_f$ to be
\begin{equation}\label{indexf}
\iota_f(z_0):=-\frac{1}{2}\deg_{z_0}\left (\frac{f_{;zz}}{|f_{;zz}|}\right),
\end{equation}
where $\deg_{z_0}$ refers to the topological degree of map from a small circle $S^1_\epsilon(z_0):=\{|z-z_0|=\epsilon\}$ to the unit circle $S^1$. Repeating the proof of Theorem \ref{PHthm}, verbatim, we arrive at the following result, which in this context is most likely known.

\begin{theorem}\label{PHthm2}
Let $X$ be a compact Riemann surface and $f\in C^\infty(X,\bR)$. If the Hessian $H_f$, with respect to any metric on $X$, has only isolated umbilical points at $z_1,\ldots,z_n\in X$, then
\begin{equation}\label{PHformula2}
\sum_{k=1}^n \iota_f(z_k)=\chi(X),
\end{equation}
where $\chi(X)=2-2g$ denotes the Euler characteristic of $X$ and $\iota_f(z_k)$ the index of the umbilical point of $H_f$ at $z_k$.
\end{theorem}

By Proposition \ref{KzzProp}, this theorem reduces to Theorem \ref{PHthm} for the CR umbilical circles on a circle bundle $M$ if we choose $f=K$, where $K$ denotes the Gaussian curvature of the metric on $X$ with K\"ahler form $-i\partial\bar \partial h$, itself the curvature of the metric $h=(\cdot,\cdot)$ on the positive, holomorphic line bundle $\pi\colon L\to X$. We note that for CR umbilical points, the metric on the line bundle defines {\it both} the metric on $X$ {\it and} the function whose Hessian is considered, resulting in a fourth order (rather than second order) partial differential operator \eqref{r2} (or, equivalently, \eqref{divP}) defining the CR umbilical locus.

\subsection{Local theory. Loewner's Conjecture}\label{Loewner} In this section, we shall let $U$ be an open subset of $X$ with a local coordinate $z$. A particular choice of metric $ds^2$ in $U$ is of course the flat metric with $\phi=0$. The Hessians $H^0_f$ of smooth, real-valued functions $f$ with respect to the flat metric, such that $z=0$ is an isolated umbilical point yield foliations near $0$ that form a subclass of all foliations by quadratic forms $\Xi= g\, dz\otimes dz$ with an isolated zero at $z=0$. In the paper \cite{SmythXavier98}, the authors refer to local foliations that are given by a flat Hessian $H^0_f$, in some coordinate system,  as ``Hessian", and also show that not all local foliations are Hessian, in this sense. The index at isolated umbilical point $z=0$ of a (flat) Hessian foliation is conjectured to be bounded above by one. This is known as Loewner's Conjecture: {\it If $f$ is a smooth function in $U$ such that its flat Hessian $H^0_f$ has an isolated umbilical point at $z=0$ (i.e., $f_{zz}=\partial^2 f/\partial z^2$ has an isolated zero at $z=0$), then the index $\iota_f(0)$ of this umbilical point is $\leq 1$.} Loewner's Conjecture is still open in general. The reader is referred, e.g., to \cite{Ivanov02} for a survey of the literature and the various approaches to this conjecture. We mention that the foliation induced by the principal line fields for embeddings of surfaces into $\bR^3$ are (flat) Hessian near its umbilical points (see, e.g., \cite{SmythXavier98}) and hence, as is well known, a proof of Loewner's Conjecture would imply, by the Poincar\'e-Hopf Index Theorem, that every embedding of 2-sphere into $\bR^3$ must have at least two umbilical points (Carath\'{e}odory's Conjecture). It is not know to the authors if the foliation near isolated zeros of $\Xi_K=e^{2\phi} K_{;zz}\; dz\otimes dz$ (i.e., near the bases of CR umbilical circles in the circle bundle $M$) is Hessian in the flat sense; if this were the case, a proof of Loewner's Conjecture would then prove the analogue of Carath\'{e}odory's Conjecture for CR umbilical points on circle bundles (by Theorem \ref{PHthm}), namely the existence of a minimum of two CR umbilical circles on a circle bundle over the Riemann sphere. We leave this as an open problem:

\begin{question}\label{QA} {\rm Is the foliation induced by the CR umbilic, quadratic form $\Xi$ constructed in Section \ref{PHsection}, or equivalently that induced by the quadratic form $\Xi_K=e^{2\phi}K_{;zz} dz\otimes dz$ where $K$ is the Gauss curvature of the metric $ds^2$ in \eqref{ds}, flat Hessian in the sense of \cite{SmythXavier98}? In other words, near an isolated zero of $K_{;zz}$ at $z=z_0$, do there exist smooth, real valued functions $f$, $\rho$ such that, in some coordinate system, it holds that
\begin{equation}\label{Hessian?}
D^2f=\rho K_{;zz}?
\end{equation}
}
\end{question}

While we do not know if the foliation corresponding to the quadratic form $\Xi_K$  is flat Hessian, we shall show, however, that the foliations given by "curved Hessians" $\Xi_f$ are not flat Hessian in general. In fact, we shall show
that: {\it For any smooth quadratic form $\Xi:=g\ dz\otimes dz$ in $U$ such that $g(0)=0$, there is a smooth function $f$ and a metric $ds^2=e^{2\phi}|dz|^2$ such that the curved Hessian form $\Xi_f$ agrees with $\Xi$ to infinite order at $z=0$.} It follows, in particular, that the index of a curved Hessian foliation can be any half-integer, and, hence, the analogue of Loewner's Conjecture for curved Hessians is not true. The statement above follows immediately from the following proposition:

\begin{proposition}\label{LoewnerProp} Let $g(z,\bar z)$ be any formal power series in $z$ and $\bar z$. Then there exist real, formal power series $f(z,\bar z)$ and $\phi(z,\bar z)$ such that $f(z,\bar z)=z+\bar z+O(|z|^2)$ and $\phi(z,\bar z)=O(|z|^2)$ such that
\begin{equation}\label{curvLoewner}
e^{2\phi}f_{;zz}=D^2f-2(D\phi)(Df)=g.
\end{equation}
\end{proposition}

\begin{proof} Let us decompose the power series into homogeneous terms
\begin{equation}\label{decomphomo}
g=\sum_{k=1}^\infty g_k,\ f=z+\bar z+\sum_{k=2}^\infty f_k,\ \phi=\sum_{k=2}^\infty \phi_k, \quad g_k\in \mathcal H_k,\ f_k,\phi_k\in \mathcal H^\bR_k,
\end{equation}
where $\mathcal H_k$ denotes the space of homogeneous polynomials of degree $k$ in $z,\bar z$ and $\mathcal H^\bR_k$ the (real) subspace consisting of those homogeneous polynomials that are real-valued. By identifying terms of degree $m$ in \eqref{decomphomo}, we find for $m=0$,
\begin{equation}\label{m=0}
D^2f_2=g_0,
\end{equation}
and then
the recursion formula
\begin{equation}\label{mrecurr}
D^2f_{m+2}-2D\phi_{m+1}=g_m +
\sum_{k=2}^{m}D\phi_{k} Df_{m+2-k},\quad m\geq 1,
\end{equation}
for $f_{m+2}\in \mathcal H^\bR_{m+2}$ and $\phi_{m+1}\in \mathcal H^\bR_{m+1}$ in terms of
$$
g_m\in \mathcal H_m,\ f_{k+1}\in \mathcal H^\bR_{k+1},\  \phi_{k}\in \mathcal H^\bR_k,\quad k\leq m.
$$
For $m=1$ the sum on the right in \eqref{mrecurr} is vacuous and, as is customary, should be regarded as a zero term.
Consider the real linear operator $T_m\colon \mathcal H^\bR_{m+2}\times \mathcal H^\bR_{m+1}\to \mathcal H_m$ defined by
\begin{equation}\label{Tm}
T_m(f_{m+2},\phi_{m+1})=D^2f_{m+2} - 2D\phi_{m+1}.
\end{equation}
We have the following:

\begin{lemma}\label{Tmlemma}
$\dim_\bR\ker T_m=3$, $m\geq 1$.
\end{lemma}

\begin{proof}[Proof of Lemma $\ref{Tmlemma}$]
We may decompose $f_{m+2}$, using the fact that $f_{m+2}$ is real-valued, as follows
\begin{equation}\label{feven}
f_{m+2}(z,\bar z)=\sum_{k=n+1}^{m+2}a_k(z^k\bar z^{m+2-k}+z^{m+2-k}\bar z^k),\quad  a_{n+1}\in \bR,\ a_{n+2},\ldots, a_{m+2}\in \bC,
\end{equation}
if $m=2n$ is even, and
\begin{equation}\label{fodd}
f_{m+2}(z,\bar z)=\sum_{k=n+2}^{m+2}a_k(z^k\bar z^{m+2-k}+z^{m+2-k}\bar z^k),\quad a_{n+2},\ldots, a_{m+2}\in \bC
\end{equation}
if $m=2n+1$ is odd. Similarly, we decompose $\phi_{m+1}$ as
\begin{equation}\label{phieven}
\phi_{m+1}(z,\bar z)=\sum_{k=n+1}^{m+1}b_k(z^k\bar z^{m+1-k}+z^{m+1-k}\bar z^k),\quad b_{n+1},\ldots, b_{m+1}\in \bC
\end{equation}
if $m=2n$ is even, and
\begin{equation}\label{phiodd}
\phi_{m+1}(z,\bar z)=\sum_{k=n+1}^{m+1}b_k(z^k\bar z^{m+1-k}+z^{m+1-k}\bar z^k),\quad  b_{n+1}\in \bR,\ b_{n+2},\ldots, b_{m+1}\in \bC,
\end{equation}
if $m=2n+1$ is odd. Note that $T_m(f_{m+2},\phi_{m+1})=D(Df_{m+2}-2\phi_{m+1})=0$ is equivalent to
\begin{equation}\label{Sm}
Df_{m+2}-2\phi_{m+1}=c\bar z^{m+1}, \quad c\in \bC.
\end{equation}
We shall first proceed under the assumption that $m=2n+1\geq 1$ is odd. Substituting \eqref{fodd} and \eqref{phiodd} into \eqref{Sm} and identifying terms of the same type $z^k\bar z^{m+1-k}$, we find the following system of equations:
\begin{align}
\label{eq1}(n+2)a_{n+2}-4b_{n+1}=0 &,\\
\label{eq2}ka_{k}-2b_{k-1}=0 &,\quad k=n+3,\ldots, m+2,\\
\label{eq3}(m+2-k)a_k-2b_{k}=0 &,\quad k=n+2,\ldots, m,\\
\label{eq4}a_{m+1}-2b_{m+1}=c &.
\end{align}
Note that if $m=1$ (i.e., $n=0$) then the set of equations in \eqref{eq3} is vacuous. It is now straightforward to verify that any choice of $b_{n+1}$ and $c$ determines uniquely the rest of the coefficients in the system of equations \eqref{eq1}-\eqref{eq4}. Also, recall that in this case, when $m=2n+1$ is odd, the coefficient $b_{n+1}$ is real. This proves Lemma \ref{Tmlemma} when $m=2n+1$ is odd. The even case (starting with $m=2$) is completely analogous; $a_{n+1}$ and $c$ determine uniquely the rest of the coefficients.  The details in this case are left to the reader.
\end{proof}

\begin{remark}\label{Tmrem}{\rm It is also easy to describe the actual kernel of $T_m$ using the proof above. For instance, if $m=2n+1$ is odd, then $\ker T_m$ is parametrized by $b_{n+1}\in \bR$ and $c\in \bC$ via \eqref{eq1}-\eqref{eq4}. One notices that it is also possible to prescribe $b_{m+1}$ instead of $c$ in \eqref{eq4}. In the even case, $m=2n$, one instead prescribes $a_{n+1}\in \bR$ and $b_{m+1}\in \bC$.
}
\end{remark}

\begin{lemma}\label{Tmsurlemma} The real linear operator $T_m\colon \mathcal H^\bR_{m+2}\times \mathcal H^\bR_{m+1}\to \mathcal H_m$, $m\geq 1$, is surjective.
\end{lemma}

\begin{proof}[Proof of Lemma $\ref{Tmsurlemma}$]
We note that
$$
\dim_\bR\mathcal H^\bR_{m+2}\times \mathcal H^\bR_{m+1}=m+3+m+2=2m+5;\quad \dim_\bR\mathcal H_{m}=2(m+1)=2m+2.
$$
The conclusion of the lemma now follows immediately from Lemma \ref{Tmlemma}.
\end{proof}

We return to the proof of Proposition \ref{LoewnerProp}. Clearly, the equation \eqref{m=0} (which is an equation of complex numbers) has a unique solution for $f_2$, once a choice of the real coefficient of $|z|^2$ has been made. Higher order terms, $f_{m+2}$ and $\phi_{m+1}$ for $m\geq 1$, are then recursively determined by \eqref{mrecurr}, which is made possible by Lemma \ref{Tmsurlemma}.
This completes the proof of Proposition \ref{LoewnerProp}.
\end{proof}

\begin{remark}\label{Loewneruniquerem} {\rm We note that the recursive solutions to \eqref{mrecurr} are not unique, but can be made so by using Remark \ref{Tmrem}. For instance, we can require $b_{m+1}=0$ at each step, which means that $\phi(z,\bar z)$ has no harmonic terms, and also specify the coefficient in front of $|z|^{2n}$ in both $f$ or $\phi$ (at alternate recursive steps, depending on the parity of $m$). This yields a unique solution. In other words, given two real power series
\begin{equation}\label{initcond}
f^0(z,\bar z)=\sum_{n=1}^\infty \alpha_n|z|^{2n},\ \phi^0(z,\bar z)=\sum_{n=1}^\infty \beta_n|z|^{2n},\quad \alpha_n,\beta_n\in \bR,
\end{equation}
there is a unique solution to \eqref{curvLoewner} such that $\phi(z,\bar z)$ has no harmonic terms (i.e., $\phi(z,0)=\phi(0,\bar z)=0$) and such that the $|z|^{2n}$ terms of $f$ and $\phi$ coincide with $f^0$ and $\phi^0$, respectively.
}
\end{remark}

\subsection{CR umbilical points on a circle bundle over a torus}\label{Torus} In this section, we shall discuss the case where the Riemann surface $X$ is a torus, $\chi(X)=0$, in which case the index formula in Theorem \ref{PHthm} allows the possibility that there is a line bundle $L$ with metric $h$ such that the unit circle bundle $M\subset L$ has {\it no} umbilical points. The authors have been unable to settle the question of the existence of such a circle bundle, and leave this as an open problem:

\begin{question}\label{QB} Does there exist a holomorphic line bundle $\pi\colon L\to X$ over a torus $X$ with a positively curved metric $h=(\cdot,\cdot)$ such that the unit circle bundle $M\subset L$ has no umbilical points?
\end{question}

We shall rephrase this question in a more explicit form, and show that if the metric is assumed to have some symmetry, then the answer is `no', i.e., the circle bundle $M$ must still have umbilical points. Let $X=\bC/\Lambda$ be a torus, where $\Lambda\subset \bC$ is a discrete lattice
\begin{equation}\label{Lambda}
\Lambda=\{m+n\omega\colon m,n\in \bZ\},
\end{equation}
for some complex number $\omega$ with $\im \omega\neq 0$. If $\pi\colon L\to X$ is a holomorphic line bundle with a positively curved metric $h=(\cdot,\cdot)$, then the curvature $i\Theta=-i\partial\bar \partial \log h$ is a global, positive $(1,1)$-form on $X$, i.e., $i\Theta=ie^{u}dz\wedge d\bar z$ for some smooth function $u(z,\bar z)$, defined on $\bC$ with periods $1$ and $\omega$. In a local trivialization of $L|_U\cong U\times \bC$ with $h(z,\bar z)=|s|_z^2$, where $s$ is a local holomorphic, non-vanishing section of $L$ in $U$, we have $e^u=-D\bar D\log h$ (denoted by $a^{-1}$ in Section \ref{prelim}). The locus of umbilical points on $M$ consists of the circles $\pi^{-1}(z_0)$, corresponding to points $z_0\in \bC$ that are zeros of the function $r=Pu$, which is smooth on $\bC$ with periods $1$ and $\omega$; where $P$ is the differential operator on $\bC$ given by \eqref{r2} or, equivalently in divergence form by \eqref{divP}. An equivalent description of the locus, by Proposition \ref{KzzProp}, is given by instead using the zero locus of the covariant derivative $K_{;zz}$, where $K$ is the Gauss curvature of the metric $ds^2$ given by \eqref{ds}. Thus, if there is a unit circle bundle $\pi\colon  M\to X=\bC/\Lambda$ without umbilical points, then there exists a smooth function $u$ on $\bC$, periodic with periods $1$ and $\omega$, such that $Pu\neq 0$ on $\bC$. We remark that $r=Pu$ is in fact obtained by applying a third-order differential operator to the function $q=Du$, and hence adding a constant to $u$, i.e., replacing $u$ by $u+C$ for some constant $C$, does not change the function $r=Pu$.

Conversely, it is known (e.g., \cite{GHbook}, Ch.\ 1.2) that if $\pi\colon L\to X=\bC/\Lambda$ is a positive line bundle and $u(z,\bar z)$ is a smooth function on $\bC$ with periods $1$ and $\omega$ such that
\begin{equation}\label{chernclass}
\frac{i}{2\pi}\int_Xe^udz\wedge d\bar z=c_1(L)\in H^2(X,\bZ)\cong \bZ,
\end{equation}
where $c_1(L)\in H^2(X,\bZ)$ denotes the first Chern class of $L$, then there is a metric $h=(\cdot,\cdot)$ on $L$ such that its curvature is given by $i\Theta=ie^{u}dz\wedge d\bar z$. Thus, if there is a smooth function $u$ on $\bC$ with periods $1$ and $\omega$ such that $Pu\neq 0$ on $\bC$, then for any positive line bundle $\pi\colon L\to X=\bC/\Lambda$, there is positively curved metric $h=(\cdot,\cdot)$ on $L$ (obtained by replacing $u$ by $u+C$, which does not change $Pu$, to achieve \eqref{chernclass} and following the procedure in \cite{GHbook}, Ch.\ 1.2) such that the unit circle bundle $M\subset L$ has no umbilical points. As mentioned above, the authors have been unable to construct such $u$ (which, if they exist, should be fairly ``generic"), but also unable to show that they cannot exist. We therefore pose this as an open question, equivalent to Question \ref{QB} above:

\begin{question}\label{QB'}
Do there exist a complex number $\omega$ with $\im \omega\neq 0$ and a smooth function $u=u(z,\bar z)$ on $\bC$ with periods $1$ and $\omega$ such that $Pu\neq 0$, where $P$ is the differential operator defined by \eqref{r2} or, equivalently, \eqref{divP}?
\end{question}

We shall conclude this section by showing that the answer to this question, and consequently that of Question \ref{QB} is `no' if we require some additional symmetry. More precisely, we shall prove:

\begin{proposition}\label{notorus} Let $\omega$ be any complex number with $\im \omega\neq 0$ and $u=u(z,\bar z)$ a smooth function on $\bC$ with periods $1$ and $\omega$. If there is a real, constant vector field $Y=\alpha\partial/\partial x+\beta\partial/\partial y$, $\alpha,\beta\in \bR$, such that $Yu=0$, then $Pu$ has zeros on $\bC$.
\end{proposition}

\begin{remark} {\rm If the vector field $Y$ does not have compact leaves on $X=\bC/\Lambda$, then of course $u$ must be constant and consequently $Pu$ vanishes identically. Thus, the proposition is only non-trivial when $Y$ has compact leaves, which happens precisely when $\alpha \im \omega/\beta$ is rational.
}
\end{remark}

\begin{proof}[Proof of Proposition $\ref{notorus}$] We shall use the divergence form of $Pu$ given by \eqref{divP} to prove that $Pu$ must have zeros. Note that $v:=e^{-u}D\bar D u$ is a real-valued function, and according to \eqref{divP} we have $e^{-2u}Pu=D(e^{-u}Dv)$. The fact that $Yu=0$ (and $Y$ has constant coefficients) implies that $Yv=0$ as well. Let $Y'$ be another real, constant vector field such that $Y$ and $Y'$ are linearly independent, and note that there are complex numbers $a,b\in \bC$ such that $D=aY+bY'$. We conclude that
$$
e^{-2u}Pu=D(e^{-u}Dv)=D(e^{-u}bY'v)=b^2Y'(e^{-u}Y'v),
$$ since $Y(e^{-u}Y'v)=0$ as well. Since $\psi:=e^{-u}Y'v$ is a smooth, real-valued function on $\bC$, periodic with periods $1$ and $\omega$ (i.e., a smooth, real-valued function on the torus $X=\bC/\Lambda$), $\psi$ has extreme points, at which $Y'\psi$ must vanish. This completes the proof that $Pu$ must have zeros when $Yu=0$.
\end{proof}

As a consequence of this proposition, we obtain the following partial answer to Question \ref{QB} under the additional assumption that $M$ has "many symmetries". Recall that $\aut(M)$ denotes the group of CR automorphisms of a CR manifold $M$.

\begin{theorem}\label{torusthm} Let $X=\bC/\Lambda$ be a torus, $\pi\colon L\to X$ a holomorphic line bundle with a positively curved metric $h=(\cdot,\cdot)$. Assume that the unit circle bundle $M\subset L$ is such that $\aut(M)$ contains a two-dimensional abelian subgroup $G$, which contains the $U(1)$-action. Then $M$ has umbilical points.
\end{theorem}

\begin{proof} We shall reduce the proof to an application of Proposition \ref{notorus}. We let $U_t$ be the generator of the $U(1)$-action and we let $\phi_s$ be a generator of a 1-parameter subgroup such that $U_t$ and $\phi_s$ generate the identity component of $G$. The fact that $G$ is abelian means that \begin{equation}\label{abelian}
U_t\circ \phi_s=\phi_s\circ U_t.
\end{equation}
The arguments in the proof of Lemma 3.17 in \cite{Epstein92} show that there is a 1-parameter subgroup $\psi_s$ of the (holomorphic) automorphism group of $X$ such that $\pi\circ \phi_s=\psi_s\circ\pi$ and such that the infinitesimal generator $Y$ of $\psi_s$ is not zero (i.e., $\psi_s$ is not constant in $s$). We let $(z,\tau)$ be coordinates in a local trivialization $L|_U\cong U\times \bC$, as in Section \ref{prelim}. We may assume that $z$ is a ``global" coordinate on $\bC$ such that $X=\bC/\Lambda$, where $\Lambda$ is the discrete lattice \eqref{Lambda} generated by $1$ and $\omega$. In these coordinates, we have $U_t(z,\tau):=(z,e^{it}\tau)$. Moreover, since $\psi_s$ is a 1-parameter subgroup of the holomorphic automorphism group of $\bC/\Lambda$, it must be of the form $\psi_s(z)=z+s\zeta$ for some $\zeta\in \bC^*$, and its infinitesimal generator $Y$ is then a real, constant vector field in $\bC$; more precisely, $Y=2\re (\zeta\partial/\partial z)$. Recall that the unit sphere bundle $M$ is given by the equation $|\tau|^2h(z,\bar z)=1$, where $h(z,\bar z)=|s_0|^2_z$ for the trivializing, local holomorphic section $s_0$ in $U$. We claim that $Yu=0$, where $u=-D\bar D\log h$. Let us assume, without loss of generality, that $h(0,0)=1$, and perform the local change of variables $\tau=e^{-iw/2}$ near $w=0$. We then obtain an equation for $M$ of the form
\begin{equation}\label{Imw=F}
\im w=F(z,\bar z),
\end{equation}
where $F(z,\bar z)=-\log h(z,\bar z)$; the point $(z,\tau)=(0,1)\in M$ now corresponds to $(z,w)=(0,0)$. Recall that $\phi_s$ is a 1-parameter family of CR automorphisms and, hence in particular, local CR diffeomorphisms on $M$ near $(0,0)$. It is well known that such mappings coincide with the smooth boundary values on $M$ of local biholomorphic mappings defined on the ``pseudoconvex side" of $M$; we shall denote, by a slight abuse of notation, these biholomorphic mappings by $\phi_s(z,w)$. We may express $\phi_s(z,w)$ as follows:
\begin{equation}\label{phi_s}
\phi_s(z,w)=(z+s\zeta,g_s(z,w)),
\end{equation}
where $g_s(z,w)$ is a 1-parameter family of holomorphic functions that extend smoothly to $M$ (and satisfy $g_0(z,w)=w$).  Using the relation \eqref{abelian}, which holds on $M$, noting that $U_t(z,\tau)=(z,e^{it}\tau)$ implies that $U_t(z,w)=(z,w-2t)$, we conclude that
\begin{equation}\label{abeliang}
g_s(z,r+iF(z,\bar z))-2t=g_s(z,r-2t+iF(z,\bar z)),
\end{equation}
where we have parametrized $M$ by $(z,r)\mapsto (z,r+iF(z,\bar z))$, $z\in \bC$ and $r\in \bR$.
Differentiating with respect to $t$ and setting $t=0$, we conclude that $\partial g_s/\partial w=1$ along $M$. Standard uniqueness properties of holomorphic functions implies that
\begin{equation}\label{k_s}
g_s(z,w)=w+k_s(z),
\end{equation}
where $k_s(z)$ is a 1-parameter family of holomorphic functions in $z$. The fact that $\phi_s$ sends $M$ to $M$ means that we have the identity
\begin{equation}\label{phiMtoM}
\im g_s(z,r+iF(z,\bar z))=F(z+s\zeta,\bar z+s\zeta).
\end{equation}
Substituting \eqref{k_s} into this, we obtain
\begin{equation}
F(z,\bar z)+\im k_s(z)=F(z+s\zeta,\bar z+s\zeta),
\end{equation}
which by differentiating with respect to $s$, setting $s=0$, and using the definition of the infinitesimal generator $Y$ shows that $YF$ coincides with a harmonic function of $z$ (namely $\im \partial k_s(z)/\partial s|_{s=0}$). It follows that $D\bar DYF=0$. Recall that $F=-\log h$ and $u=-D\bar D\log h$, and hence we have
$$
Yu=Y(-D\bar D\log h)=D\bar D Y(-\log h)=D\bar D YF=0,
$$
as claimed above. Now, it follows from Proposition \ref{notorus} that $Pu$ must have zeros, and hence $M$ has umbilical points. This completes the proof of Theorem \ref{torusthm}.
\end{proof}

\subsection{Locally spherical circle bundles} For completeness, we shall also briefly discuss locally spherical (totally umbilical) circle bundles. A straightforward consequence of Proposition \ref{KzzProp} and a theorem of Calabi is that there are locally spherical, unit circle bundles $M$ in every positive line bundle $\pi\colon L\to X$ over a compact Riemann surface, and these are essentially unique. Before stating this result more precisely, let us recall that if $\omega$ is a positive $(1,1)$-form on $X$ such that
$$
\int_X\omega=1,
$$
then there is a unique metric $(\cdot,\cdot)$ on the positive line bundle $\pi\colon L\to X$ such that its curvature form satisfies
$$
\frac{i}{2\pi}\,\Theta=c_1(L)\,\omega,
$$
where $c_1(L)\in H^2(X,\bZ)$ as above denotes the first Chern class of $L$. Also, recall that on a compact Riemann surface of genus $g$, there is a unique metric, normalized so that $X$ has unit area, for which the curvature is constant, $K=2-2g=\chi(X)$.

\begin{theorem}\label{Calabi} Let $X$ be a compact Riemann surface and let
$\omega_0$ denote the K\"ahler form of the constant curvature metric on $X$, normalized so that $X$ has unit area. If $\pi\colon  L\to X$ is a positive holomorphic line bundle on $X$, then the unit sphere bundle $M$ with respect to a positively curved metric $(\cdot,\cdot)$ on $L$ is locally spherical if and only if the curvature form of the metric satisfies
\begin{equation}\label{constcurv}
\frac{i}{2\pi}\,\Theta=c_1(L)\,\omega_0,
\end{equation}
where $c_1(L)\in H^2(X,\bZ)$ denotes the first Chern class of $L$.
\end{theorem}

\begin{proof} If we unwind the definitions in the preliminaries above, we see that in the local trivialization the curvature of a metric $(\cdot,\cdot)$ on $L$ is given by
$$
\frac{i}{2\pi}\,\Theta= \frac{i}{2\pi}\,(-D\bar D \log h)\,  dz\wedge d\bar z=\frac{i}{2\pi}\,a^{-1}dz\wedge d\bar z=\frac{1}{\pi}\,\omega,
$$
where $\omega$ is the K\"ahler form of the metric
$$
ds^2:=e^{2\phi}|dz|^2,\quad 2\phi:=-\log a=u.
$$
By Proposition \ref{KzzProp}, the umbilical locus on $M$, projected on $X$, coincides with the zero locus of $K_{;zz}$, where $K$ is the curvature of $ds^2$. Thus, $M$ is locally spherical if and only if $K_{;zz}$ vanishes identically on $X$.

The proof that \eqref{constcurv} implies that $M$ is locally spherical is now immediate, since the corresponding $K$, by definition, is constant and, hence, $K_{;zz}\equiv 0$. The converse follows also immediately from a theorem of Calabi,  Theorem 3.1 in \cite{Calabi82}, which asserts that the only solutions to the equation $K_{;zz}= 0$ on a compact Riemann surface are the constant curvature metrics.
\end{proof}

\section{Proofs of the main results}\label{Proofs}

In this section, we complete the proofs of the results stated in the introduction. We first remind the reader that a compact, strictly pseudoconvex, three-dimensional CR manifold with a transverse, free CR $U(1)$-action is CR equivalent to a unit circle bundle $M$ in a holomorphic line bundle $\pi\colon L\to  X\cong M/U(1)$ with a positively curved metric $h=(\cdot,\cdot)$; see \cite{Epstein92}, Lemma A4. Consequently, Theorem \ref{Main1} follows immediately from Theorem \ref{PHthm}. Next, a circular hypersurface $M$ in $\bC^2$ as in Theorem \ref{MainCirc} clearly has a transverse, free CR $U(1)$-action and in this case $X=\bC\bP^1$ . Indeed, such $M$ can be realized as a unit circle bundle over $\bC\bP^1$ explicitly by blowing up the origin in $\bC^2$; see e.g. \cite{BlandDuchamp91}. Consequently, Theorem \ref{MainCirc} follows from Theorem \ref{Main1}. We are left with proving Theorem \ref{Maintorus}.

\begin{proof}[Proof of Theorem $\ref{Maintorus}$] In view of Theorem \ref{Main1}, it suffices to prove Theorem \ref{Maintorus} under the assumption that $X\cong M/U(1)$ is a torus. Let $\aut_0(M)$ denote the identity component of $\aut(M)$. If $\aut_0(M)$ is non-compact, then $M$ is CR equivalent to the standard sphere \cite{Lee96}, and is hence totally umbilical (and also cannot be a circle bundle over a torus). Thus, we may assume that $\aut_0(M)$ is compact. We identify two cases:
\begin{itemize}
\item[1.] $\dim_\bR\aut_0(M)\geq 3$. In this case, it follows from the work of E. Cartan \cite{Cartan33} that $M$ is homogeneous and must be on his ``classification list". As mentioned in the introduction, the only compact, non-umbilical homogeneous CR manifolds in Cartan's classification are $\mu_\alpha$ and their covers. Neither can be a unit circle bundle over a torus for topological reasons (since their fundamental groups are finite).
\item[2.] $\dim_\bR\aut_0(M)=2$. A compact, connected Lie group $G=\aut_0(M)$ of dimension 2 is necessarily abelian. This is well known, but the authors have been unable to find a direct reference for this statement, and therefore sketch here briefly a proof for the reader's convenience. First, we note that there are only two Lie algebras, up to isomorphism, of dimension 2, namely the abelian one $\frak g_1$ and another $\frak g_2$ that can be generated by elements $\xi_1,\xi_2$ with $[\xi_1,\xi_2]=\xi_2$. Next, any compact, connected Lie group $G$ is unimodular \cite{HelgaBook}, which implies that the differential of the modular function vanishes, i.e., $\tr\ad_\xi=0$ for all $\xi$ in the Lie algebra $\frak g$. Now, clearly we have $\tr\ad_{\xi}\neq 0$ for $\xi=\xi_1\in \frak g_2$ as above. We conclude that the Lie algebra of a 2-dimensional compact, connected Lie group $G$ is isomorphic to $\frak g_1$ and hence $G$ is abelian. The conclusion of the theorem now follows from Theorem \ref{torusthm} above.
\end{itemize}
These two cases exhaust the remaining possibilities and hence Theorem \ref{Maintorus} is proved.
\end{proof}


\def\cprime{$'$}

\end{document}